\documentclass[12pt,reqno]{amsart}
\usepackage{url}
\usepackage{amstext}
\usepackage{a4wide}
\usepackage{graphicx}
\usepackage{bm}
\allowdisplaybreaks \numberwithin{equation}{section}
\usepackage{color}
\usepackage{cases}
\usepackage{amsfonts,amsmath,latexsym,verbatim,amscd,mathrsfs,color,array}
\usepackage[colorlinks=true]{hyperref}
\usepackage[hmargin=3cm, vmargin=3cm]{geometry}

\usepackage{amsmath,amssymb,amsthm,amsfonts,graphicx,color}
\usepackage{amssymb,wasysym}
\usepackage{pdfsync}
\usepackage{epstopdf}
\usepackage{amsfonts,amsmath,latexsym,verbatim,amscd,mathrsfs,color,array}
\usepackage[numbers,sort&compress]{natbib}

\makeatletter
\def\@currentlabel{2.1}\label{e:dispaa}
\def\@currentlabel{2.21}\label{e:dispau}
\def\@currentlabel{2.22}\label{e:dispav}
\def\@currentlabel{2.23}\label{e:dispaw}
\def\@currentlabel{2.24}\label{e:dispax}
\def\theequation{\thesection.\@arabic\c@equation}
\makeatother

\numberwithin{equation}{section}

\newtheorem{theorem}{Theorem}[section]

\newtheorem{corollary}[theorem]{Corollary}
\newtheorem{lemma}[theorem]{Lemma}

\theoremstyle{definition}

\theoremstyle{remark}
\newtheorem{remark}[theorem]{Remark}
\newtheorem{remarks}[theorem]{Remarks}

\newcommand{\ep}{\varepsilon}

\newcommand{\la}{\lambda}

\newcommand{\R}{\mathbb{R}}

\begin{document}
    
    \title
    {Traveling vortex pairs for 2D  Boussinesq equations}

    \author{Daomin Cao, Shanfa Lai, Guolin qin}
    
    \address{Institute of Applied Mathematics, Chinese Academy of Sciences, Beijing 100190, and University of Chinese Academy of Sciences, Beijing 100049,  P.R. China}
    \email{dmcao@amt.ac.cn}
    
    \address{Institute of Applied Mathematics, Chinese Academy of Sciences, Beijing 100190, and University of Chinese Academy of Sciences, Beijing 100049,  P.R. China}
    \email{laishanfa@amss.ac.cn}
    
    \address{Institute of Applied Mathematics, Chinese Academy of Sciences, Beijing 100190, and University of Chinese Academy of Sciences, Beijing 100049,  P.R. China}
    \email{qinguolin18@mails.ucas.ac.cn}

    \begin{abstract}
       In this paper, we study the existence and asymptotic properties of the traveling vortex pairs for the two-dimensional inviscid incompressible Boussinesq equations. We construct a family of traveling vorticity pairs, which constitutes the de-singularization of a pair of point vortices with equal intensity but opposite sign. Using the improved vorticity method, we also give limiting position of the supports of vorticities.
        ~\\
        
    \end{abstract}
    \maketitle
    \section{Introduction and main results}\label{sect 1}
   In this paper we are concerned with the two-dimensional Boussinesq equations
       \begin{numcases}
        {}
        \label{0-1} \partial_t \mathbf{v} + (\mathbf{v} \cdot\nabla) \mathbf{v} + \nabla P = \nu \Delta \mathbf{v} + \rho \mathbf{e}_2, &\\
        \label{0-2}  \partial_t \rho + (\mathbf{v} \cdot\nabla) \rho= \vartheta \Delta \rho,  &\\
        \label{0-3}\nabla\cdot \mathbf{v} = 0.&
    \end{numcases}
where $\mathbf{v}=(v_1,v_2)$ is the velocity field, $P$ is the scalar pressure, and $\rho$ is the density of the fluid.  In the above equation, $ \nu\geq 0$ is the viscosity coefficient, $ \vartheta\geq 0 $ is the thermal conductivity coefficient, and $\rho \mathbf{e}_2$ is the buoyancy force.
    
    The Boussinesq equations has very important mathematical implications, especially in two dimensions. In fact, the two-dimensional Boussinesq system retains some key properties of the three-dimensional Euler equations and the three-dimensional Navier-Stokes equations, such as the vortex stretching mechanism. There are many studies of the global/local well-posedness for this system, e.g., \cite{BS,Ch,CN,EW, LB,TWZZ} and the references therein. The regularity and finite time blow up of solutions are well-known open problems for two-dimensional inviscid Boussinesq system $(\nu\equiv \vartheta \equiv 0 )$. Recently, T. Hmidi and K. Choi et al. have done some work on such issues for reference \cite{CKY,KS,KT,KY,Hmidi2}.

    In this paper, we consider hydrostatic equilibrium of the 2D inviscid Boussinesq system $(\nu\equiv \vartheta \equiv 0 )$  in the whole plane. Although there are many results for the above unsteady cases, there seems to be little work for the relative equilibrium cases.  Our main aim is to construct a family of traveling vortex pairs for two-dimensional Boussinesq equations with general vorticity and density functions, which constitute the desingularization of a pair of point vortices with equal strength but opposite sign.
    
    Now we deduce the vorticity-density formulation of the system $(\nu\equiv \vartheta \equiv 0 )$.  Let $\omega=\nabla \times \mathbf{v}=\partial_{1} v_{2}-\partial_{2} v_{1}$ be the corresponding vorticity of the flow. By virtue of Eq. \eqref{0-3},  there exists a Stokes stream function $\Psi(x, t)$ such that
    $$
    \mathbf{v}=\nabla^{\perp} \Psi:=\left(\partial_{x_{2}} \Psi,-\partial_{x_{1}} \Psi\right).
    $$
    Note that by the definitions of $\Psi$ and $\omega$ we have $-\Delta \Psi=\omega$. Thus the velocity field $\mathbf{v}$ can be recovered from the vorticity function $\omega$ through the Biot-Savart law
    \begin{equation*}
        \mathbf{v}=\nabla^{\perp}(-\Delta)^{-1} \omega=-\frac{1}{2 \pi} \frac{x^{\perp}}{|x|^{2}} * \omega.
    \end{equation*}
    Taking the curl of \eqref{0-1}, we get the following vorticity-density formulation:
     \begin{numcases}
        {}
        \label{0-4} \partial_t {\omega} + (\mathbf{v} \cdot\nabla) \omega = \partial_{{1}}\rho, &\\
        \label{0-5} \partial_t \rho + (\mathbf{v} \cdot\nabla) \rho=0, &\\
        \label{0-6}\mathbf{v}=\nabla^{\perp} \Psi, \Psi=(-\Delta)^{-1} \omega.&
    \end{numcases}
       
      Vorticity localization in nature consists of a finite number of small regions, outside which vorticity either is zero or rapidly decreases to zero. Mathematically we can consider the singular solution of the Boussinesq system
       \begin{equation}\label{0-7}
           \omega=\sum_{i=1}^{N} \kappa_{i} \delta\left(x-\Im_{i}(t)\right).
       \end{equation}
   The corresponding velocity can be described as 
       \begin{equation*}
           \mathbf{v}(x)=-\sum_{i=1}^{N} \frac{\kappa_{i}}{2 \pi} \frac{\left(x-\Im_{i}(t)\right)^{\perp}}{\left|x-\Im_{i}(t)\right|^{2}},
       \end{equation*}
    and the position of the vortices $\Im_{i}: \mathbb{R} \rightarrow \mathbb{R}^{2}$ satisfies
    \begin{equation}\label{0-8}
          \frac{d}{d t} \Im_{i}(t)=-\sum_{j=1, j \neq i}^{N} \frac{\kappa_{j}}{2 \pi} \frac{\left(\Im_{i}(t)-\Im_{j}(t)\right)^{\perp}}{\left|\Im_{i}(t)-\Im_{j}(t)\right|^{2}} .
    \end{equation}
      Here $\delta(x)$ denotes the standard Dirac mass at the origin, and the constants $\kappa_{i}$ are called the intensities of the vortices $\Im_{i}(t)$. The Hamiltonian system \eqref{0-8} is now generally referred to as the point vortex model or the Kirchhoff-Routh model (see \cite{Mar}).
       
     In this paper, we study solutions with sharply concentrated vorticity around the points $\Im_{1}(t), \ldots, \Im_{N}(t)$ which are idealized as regular solutions that approximate the point vortex model. Finding these regular solutions is a classical vortex desingularization problem, see \cite{Da1,Cao1,Cao2,Cao3} and the references therein.
     Under certain appropriate conditions, the system \eqref{0-8} allows a simple traveling solution of the form 
       \begin{equation}\label{0-9}
           \Im_{i}(t)=b_{i}+W t \mathbf{e}_{1}, \quad i=1, \ldots, N,
       \end{equation}
       where $b_{1}, \ldots, b_{N}$ are points in $\mathbb{R}^{2}$, the constant $W>0$ is the speed, and $\mathbf{e}_{1}=(1,0)$. In this paper, we focus on traveling solutions of the form \eqref{0-9}. We consider the traveling vortex pair, namely the solution with $N=2$, and   
    \begin{equation}\label{0-10}
        b_{1}=-b_{2}=r^* \mathbf{e}_{2}, \quad \mathbf{e}_{2}=(0,1), \quad \kappa_{1}=-\kappa_{2}=\kappa, \quad W>0
    \end{equation}
    where  $W, \kappa>0$  are  two given constants and  $r^*>0$ will be suitably chosen later. It is easy to check that such a vortex pair is indeed a solution of system \eqref{0-8}.
    
    To study vortex desingularization problem of the traveling vortex pair, we can look for traveling solutions to \eqref{0-4}-\eqref{0-6} by requiring that
    \begin{equation}\label{0-11}
        \omega(x,t)=\zeta(x_1-Wt,x_2),\  \rho(x,t)=\eta(x_1-Wt,x_2),
        \end{equation}
    for some profile functions $(\zeta, \eta)$ defined on $\mathbb{R}^{2}$. 
   Let  $\Psi\left(x_{1}, x_{2}, t\right)=\phi\left(x_{1}-W t, x_{2}\right)$ with $\phi=(-\Delta)^{-1} \zeta$. Thus the system \eqref{0-4}-\eqref{0-6} is reduced to a stationary problem  
\begin{numcases}
    {}\label{0-12} \nabla^{\perp}(\phi-Wx_2)\cdot\nabla \zeta=\partial_{1}\eta, &\\
    \label{0-13} \nabla^{\perp}(\phi-Wx_2)\cdot\nabla \eta=0,&\\
    \label{0-14} \phi=(-\Delta)^{-1} \zeta.&
\end{numcases}
We consider axisymmetric flows about  $x_{1}$, and limit our attention henceforth to the half-plane $\Pi=\left\{x \in \mathbb{R}^{2} \mid x_{2}>0\right\}$. Specifically, we may assume
\begin{equation*}
    \zeta\left(x_{1}, x_{2}\right)=-\zeta\left(x_{1},-x_{2}\right),
\end{equation*}
so that $\phi\left(x_{1}, x_{2}\right)=-\phi\left(x_{1}, -x_{2}\right)$.
Suppose that there exists some function $b(s)$ satisfying
\begin{equation}\label{0-15}
     \eta=b(\phi-Wx_2).
\end{equation}
Taking \eqref{0-15} into \eqref{0-12}, we get
 \begin{equation*}
     \nabla^{\perp}(\phi-Wx_2)\cdot \nabla (\zeta+x_2b'(\phi-Wx_2))=0.
 \end{equation*}
 Thus, if there exists some function $a(s)$ such that
 \begin{equation*}
     \zeta=a(\phi-Wx_2)-x_2b'(\phi-Wx_2)
 \end{equation*}
then \eqref{0-12} is satisfied automatically. Thus, we are led to the following nonlinear elliptic equation:
\begin{equation}\label{0-16}
        -\Delta\phi=\lambda(f(\phi-Wx_2-\mu)+x_2g(\phi-Wx_2-\mu)) \  \text{in} \ \Pi,
\end{equation}
 For technical reasons, we make the following assumptions on $f$ and $g$.

\begin{enumerate}
    \item[(H1)]  $f \in C(\mathbb{R}\setminus\{0\})$ and $g \in C(\mathbb{R})$ are both nonnegative and nondecreasing.
\end{enumerate}

Set
$$
i(x_2, t)=f(t)+x_2g(t)
$$
and $I(x_2, t)=\int_{0}^{t} i(x_2, s) d s$. Let $J(x_2, \cdot)$, the modified conjugate function to $I(x_2, \cdot)$, be defined by
\begin{equation}\label{0-17}
    J(x_2, s)=\sup _{t}[s t-I(x_2, t)]\  \text { if } s \geq 0 ;\  J(x_2, s)=0 \text { if } s<0 .
\end{equation}

We further require the following:
\begin{enumerate}
    \item [(H2)] For each $x_2>0, i(x_2, t)=0$ if $t \leq 0$, and $i(x_2, t)$ is strictly increasing (w.r.t. t) in $[0,+\infty)$.
    \item [(H3)] For each $d>0$, there exist $\delta_{0} \in(0,1)$ and $\delta_{1}>0$ such that
    $$
    I(x_2, t) \leq \delta_{0} i(x_2, t) t+\delta_{1} i(x_2, t), \quad \forall\ t>0, \forall\ 0<x_2 \leq d
    $$
    
     \item[(H4)]For all $\tau>0$, there holds
     $$
     \lim _{t \rightarrow+\infty} i(x_2, t) e^{-\tau t}=0, \quad \forall\ x_2>0 .
     $$ 
\end{enumerate}
\begin{remarks}
    We give some remarks on the above assumptions.
    \begin{enumerate}
        \item Assumption (H2) excludes the cases where $g\equiv0$ and $f$ is a step function. This case has been studied in \cite{Cao1, Cao4,FT}.
        \item Assumption (H3) is of Ambrosetti-Rabinowitz type, cf. condition $(p5)$ in \cite{AR}. This
        assumption implies that $ i(r, t) \rightarrow+\infty \text { as } t \rightarrow+\infty$ (see\cite{Ni}). Note that $i(x_2, \cdot)$ and $\partial_{s} J(x_2, \cdot)$ are inverse graphs (see \cite{Roc}),  one can check
        that (H3) is actually equivalent to   
        \item [$(H3)'$] : $$ J(x_2,s)\ge (1-\delta_{0})\partial_{s}J(x_2,s)s-\delta_{1}s, \forall\ t>0, \forall\ 0<x_2\le d.$$ 
         \item The assumption (H4) relaxes the condition (1.6) in \cite{Ber}. In such problems, proper assumptions on the  growth at infinite are usually necessary.
    \end{enumerate}
\end{remarks}
 
 It is not difficult to find some $(f,g)$ that satisfies hypothesis (H1)-(H4). For example, Lamb found an explicit solution of \eqref{0-16} in \cite{Lamb}, considering the case $f(s)=s_{+}^{p}$, $g(s)\equiv 0$ and $p=1$. The explicit solution obtained in \cite{Lamb} is often called the Lamb dipole or Chaplygin-Lamb dipole, see  \cite{Mel}. For the type above,  de Valeriola and Van Schaftingen in \cite{DV} investigate the case where $p>1$. Turkington also gives a desingularization solution in \cite{Tur89} for $ f(s)=\alpha \chi_{\{s>0\}}, g(s)=s_{+}$, which can be taken as $p=0$.

Exact solutions of \eqref{0-16} are known only in special cases. Besides those exact solutions, the existence of steady vortex pairs has been rigorously proved. Turkington  studied in \cite{Tur83} when $f$ is the Heaviside function and constructed a family of desingularization solutions for \eqref{0-16}. Burton \cite{Bu0} and Badiani \cite{Bad} studied the existence of steady vortex pairs with a prescribed distribution of the vorticity. Inspired by Burton’s work, Elcrat and Miller \cite{Elc,Elc1} extended Turkington’s result to the context of rearrangements of vorticity and to more general exterior plane domains.
Recently, Hmidi and Mateu \cite{Hmidi} established the existence of corotating and counter-rotating vortex pairs of simply connected patches for the active scalar equations by using the contour dynamics equations. For numerical studies, see, e.g., \cite{Del,Kiz}.

As mentioned above, for some special nonlinearities $f$, there are already some desingularization results.
 To the best of our knowledge, no related paper has established the solution of the two-dimensional Boussinesq equations to the desingularization of traveling  vortex pairs of the general vorticity and density functions. Our goal is to generalize these special cases to more general nonlinear terms and thus enrich the known solutions to pairs of steady vortices.

Now, we turn to state the main results. To do this, let's introduce some known facts. Let
 \begin{equation*}
    J_{G}(s)=\sup _{t}\left[s t-G(t)\right]\  \text { if } s \geq 0 ;\  J_{G}(s)=0 \text { if } s<0,
\end{equation*}
where $G(t)=\int_{0}^{t}g(s)ds$. We introduce the limiting profile $V^{\kappa}: \mathbb{R}^{2} \rightarrow \mathbb{R}$ defined as the unique radially symmetric solution of the problem
\begin{equation}
    \left\{\begin{array}{l}
        -\Delta V^{\kappa}=g(V^{\kappa}), \quad x \in \mathbb{R}^{2}, \\
        \int_{\mathbb{R}^{2}}g(V^{\kappa}) d x=\kappa.
    \end{array}\right.
\end{equation}
According to the Pohozaev identity, $V^{\kappa}$ has the following properties
\begin{equation}
    \int_{B_{r}(0)}G(V^{\kappa}(z))dz=\frac{\pi}{2}((V^{\kappa})'|_{\partial B_{r}(0)})^2 r^2, \ \ \ \int_{B_{r}(0)}g(V^{\kappa})dz=2\pi(-(V^{\kappa})'|_{\partial B_{r}(0)})r,
\end{equation}
where $supp(V_{+}^{\kappa})=B_{r}(0)$. By the above formula, we can calculate directly
\begin{equation}
    \begin{split}
        \mathcal{C}_{g, \kappa}:&=\int_{\R^2}\Big(V^{\kappa}(z)g(V^{\kappa}(z))-J_{G}\Big(g(V^{\kappa}(z))\Big)\Big)dz\\
        &=\int_{\R^2}G(V^{\kappa}(z))dz=\frac{\kappa^2}{8\pi}.
    \end{split}
\end{equation}
And from this we find an interesting fact that $\mathcal{C}_{g,\kappa}$ does not change with the choice of $g$.

We similarly define $J_{F}$ as a convex conjugate of $F$, where $F(t)=\int_{0}^{t}f(s)ds$, and introduce the function $U^{\kappa}:\R^2 \to \R$ defined as the unique radially symmetric solution of the problem
\begin{equation}
    \left\{\begin{array}{l}
        -\Delta U^{\kappa}=f(U^{\kappa}), \quad x \in \mathbb{R}^{2} \\
        \int_{\mathbb{R}^{2}}f(U^{\kappa}) d x=\kappa.
    \end{array}\right.
\end{equation}
Let
\begin{equation}\label{2-18}
    \mathcal{W}(t)=\frac{\kappa^2}{4\pi}\log\frac{1}{2t}+{\kappa Wt}+\int_{\R^2}\frac{1+\alpha r^{*}}{1+\alpha t} J_{F}\left(\frac{1+\alpha t}{1+\alpha r^{*}}f(U^{\kappa}(z))\right)dz,\ \ t>0.
\end{equation}
where $r^*$ is the unique solution of \eqref{2.1} on $(0, \infty)$.

 Our first main result concerns the special case where $g=\alpha f$ for some $\alpha\geq 0$ and read as follows.

\begin{theorem}\label{thm1}
 Let $\alpha>0$ be given, $W$ and $\kappa>0$ satisfy $ W\ge \frac{(2-\alpha)_{+}^2\kappa}{32 \pi}$. Consider the traveling vortex pair given by \eqref{0-10}. Suppose that $(f,\alpha f)$ satisfies (H1)-(H4). Then for $\lambda>0$ large there is a traveling solution $\left(\omega^{\lambda}, \rho^{\lambda}, \Psi^{\lambda}\right)$ to the Boussinesq system \eqref{0-4}-\eqref{0-6} given by \eqref{0-11} and
 \begin{equation*}
     \begin{split}
         \omega^\la (x,t)=\zeta^{\lambda}\left(x_{1}-Wt, x_{2}\right)&=\lambda(1+\alpha x_2) f\left(\Psi^{\lambda}\left(x_{1}-Wt, x_{2}\right)-W x_{2}-\mu^{\lambda}\right),\\
         \rho^\la(x,t)=\eta^{\lambda}(x_1-Wt,x_2)&=-\lambda \alpha F(\Psi^{\lambda}(x_{1}-Wt, x_{2})-Wx_2-\mu^{\lambda}),
     \end{split}
 \end{equation*}
for all $\left(x_{1}, x_{2}\right) \in \Pi$,  which satisfies
\begin{equation}\label{0-20}
    \begin{split}
         &supp(\zeta^{\lambda})\cap \Pi=supp(\eta^{\lambda})\cap \Pi \subset B_{L\lambda^{-\frac{1}{2}}}(x^{\lambda}),\ \text{for some}\ x^{\lambda}=(0,x_2^{\lambda})\in \Pi, \\
        &\zeta^{\lambda}(x) \rightharpoonup \kappa \delta\left(x-b_{1}\right)-\kappa \delta\left(x-b_{2}\right),\ \ b_1=-b_2=r^* \mathbf{e}_{2},\  \ \text { as } \lambda \rightarrow+\infty,
    \end{split}
\end{equation}
where the convergence is in the sense of measures and $L>0$ is a constant independent of $\lambda$. Moreover, one has
\begin{equation}\label{0-21}
   \zeta^{\lambda}\left(x_{1}, x_{2}\right)=-\zeta^{\lambda}\left(x_{1}, -x_{2}\right)=\zeta^{\lambda}\left(-x_{1},x_{2}\right),\ \ \int_{\Pi} \zeta^{\lambda}dx=\kappa,
\end{equation}
and    \begin{equation*}
     \frac{d \mathcal{W}(x_2^{\lambda})}{dt}\to 0,\ \  x_2^{\lambda}\to r^*, \ \ \text{as} \ \lambda \to \infty. 
\end{equation*}

\end{theorem} 
\begin{remarks}
    For $ \alpha=0$, the above theorem is reduced to a desingularization result for the Euler equation. One can see from the definition of $r^*$ that the limiting position is dependent not only on $W$ and $\kappa$ but also on the vorticity-density function $(f, g)$. This  is a special phenomenon of the Boussinesq system which does not occur for the Euler equation.
\end{remarks}

Recently, D. Cao, J. Wan and W. Zhan \cite{Cao5} has obtained a result of desingularization for the case of $f(s)=g(s)=s_{+}^{p}$  in a bounded domain. Using the above theorem, we can  verify that such special $(f,g)$ satisfies all assumptions (H1)-(H4) and hence as a corollary, we can easily get  the existence of traveling vortex pairs with the same density function $(f,g)$ as \cite{Cao5}.

 Our next two results are concerned with the case where $f$ and $g$ are both general functions but one of them is vanishing as $\la \to +\infty$.
\begin{theorem}\label{thm2}
    Suppose that $f(0^{+})=0$, $f(s)>0$ for $s>0$ and $(f, \delta_{\lambda}g)$ satisfies (H1)-(H4) with the constants $\delta_0$ and $\delta_1$ independent of $\lambda$ in (H3), where $ \delta_\lambda >0 \ \text{and}\ \delta_{\lambda}\to 0 \ \text{as} \ \lambda \to \infty $.  Then for $\lambda>0$ large there is a traveling solution $\left(\omega^{\lambda}, \rho^{\lambda}, \Psi^{\lambda}\right)$ to the Boussinesq system \eqref{0-4}-\eqref{0-6} given by \eqref{0-11} and
     \begin{equation*}
        \begin{split}
            \zeta^{\lambda}\left(x_{1}-Wt, x_{2}\right)&=\lambda\Big( f\left(\Psi^{\lambda}\left(x_{1}-Wt, x_{2}\right)-W x_{2}-\mu^{\lambda}\right)\\
            &\ \ \ \ \ \ \ \ \ \ \ \ \ \ +x_2 \delta_{\lambda} g\left(\Psi^{\lambda}\left(x_{1}-Wt, x_{2}\right)-W x_{2}-\mu^{\lambda}\right)\Big),\\
            \eta^{\lambda}(x_1-Wt,x_2)&=-\lambda \delta_{\lambda} G(\Psi^{\lambda}(x_{1}-Wt, x_{2})-Wx_2-\mu^{\lambda}),
        \end{split}
    \end{equation*}
for all $\left(x_{1}, x_{2}\right) \in \Pi$, which satisfies \eqref{0-21} and
\begin{equation}
    \begin{split}
        &supp(\zeta^{\lambda})\cap \Pi=supp(\eta^{\lambda})\cap \Pi \subset B_{L\lambda^{-\frac{1}{2}}}(x^{\lambda}),\ \text{for some}\ x^{\lambda}=(0,x_2^{\lambda})\in \Pi, \\
        &\zeta^{\lambda}(x) \rightharpoonup \kappa \delta\left(x-b_{1}\right)-\kappa \delta\left(x-b_{2}\right),\ \ b_1=-b_2=r_1^* \mathbf{e}_{2},\  \ \text { as } \lambda \rightarrow+\infty.
    \end{split}
\end{equation}
where the convergence is in the sense of measures and $L>0$ is a constant independent of $\lambda$.
 Moreover, one has
 \begin{equation*}
     \frac{d \mathcal{W}_1(x_2^{\lambda})}{dt}\to 0, \ x_2^{\lambda} \to r_1^*,\ \ \text{as}\ \lambda \to \infty.
 \end{equation*}
where $r_1^*=\frac{\kappa}{4 \pi W}$ and $\mathcal{W}_1$ be defined as follows
 $$ \mathcal{W}_1(t):=\frac{\kappa^2}{4\pi}\log\frac{1}{2t}+{\kappa Wt}.$$
   \end{theorem}
\begin{theorem}\label{thm3}
    Suppose that $W>0$ and $\kappa>0$. Let $g\in C^1(\R^2)$, $g'(s)>0$ for $s>0$ and $(\delta_{\lambda}f, g)$ satisfies (H1)-(H4) with the constants $\delta_0$ and $\delta_1$ independent of $\lambda$ in (H3),  where $ \delta_\lambda >0 \ \text{and}\ \delta_{\lambda}\to 0 \ \text{as} \ \lambda \to \infty $. Then for $\lambda>0$ large there is a traveling solution $\left(\omega^{\lambda}, \rho^{\lambda}, \Psi^{\lambda}\right)$ to the Boussinesq system \eqref{0-4}-\eqref{0-6} given by \eqref{0-11} and 
     \begin{equation*}
        \begin{split}
            \zeta^{\lambda}\left(x_{1}-Wt, x_{2}\right)&=\lambda\Big( \delta_{\lambda} f\left(\Psi^{\lambda}\left(x_{1}-Wt, x_{2}\right)-W x_{2}-\mu^{\lambda}\right)\\
            &\ \ \ \ \ \ \ \ \ \ \ \ \ \ +x_2  g\left(\Psi^{\lambda}\left(x_{1}-Wt, x_{2}\right)-W x_{2}-\mu^{\lambda}\right)\Big),\\
            \eta^{\lambda}(x_1-Wt,x_2)&=-\lambda G(\Psi^{\lambda}(x_{1}-Wt, x_{2})-Wx_2-\mu^{\lambda}),
        \end{split}
    \end{equation*}
    for all $\left(x_{1}, x_{2}\right) \in \Pi$, which satisfies \eqref{0-21} and
    \begin{equation}
        \begin{split}
            &supp(\zeta^{\lambda})\cap \Pi=supp(\eta^{\lambda})\cap \Pi \subset B_{L\lambda^{-\frac{1}{2}}}(x^{\lambda}),\ \text{for some}\ x^{\lambda}=(0,x_2^{\lambda})\in \Pi, \\
            &\zeta^{\lambda}(x) \rightharpoonup \kappa \delta\left(x-b_{1}\right)-\kappa \delta\left(x-b_{2}\right),\ \ b_1=-b_2=r_2^* \mathbf{e}_{2},\  \ \text { as } \lambda \rightarrow+\infty.
        \end{split}
    \end{equation}
    where the convergence is in the sense of measures and $L>0$ is a constant independent of $\lambda$.
    Moreover, one has
    \begin{equation*}
        \frac{d \mathcal{W}_2(x_2^{\lambda})}{dt}\to 0, \ x_2^{\lambda} \to r_2^*,\ \ \text{as}\ \lambda \to \infty.
    \end{equation*}
    where  $r
    _2^*=\frac{\kappa}{8\pi W}$ and $\mathcal{W}_2$ is defined as 
    \begin{equation}
        \mathcal{W}_{2}(t)=\frac{\kappa^2}{4\pi}\log\frac{1}{2t}+{\kappa Wt}+  \frac{r_2^*}{t}\int_{\R^2}J_{G}\Big(\frac{t}{r_2^*}g(V^{\kappa}(z))\Big)dz.
    \end{equation}
\end{theorem}

    \begin{remarks}
        What we want to emphasize is that it's not a necessary condition for $g\in C^1(\R^2)$. For example, for some special functions $g(t)=1_{\{t>0\}}$, $\kappa=\pi$, we can calculate $\int_{\R^2}G(V^{\kappa}(z))dz=\frac{\kappa^2}{8\pi}$, but for $s$ greater than $1$, we have $J_{G}(s)=+\infty$, so the integral in $C_{g,\kappa}$ may diverge. However, we also believe that this type of characteristic function can be obtained by the method with some modifications in this paper.
     \end{remarks}

    \section{Proof of Theorem \ref{thm1}  }\label{sect2}
    In this section, we will give the proof for Theorem \ref{thm1}. Let $G(\cdot,\cdot)$ be the Green function for $-\Delta$ in $\Pi$ with zero Dirichlet data, namely,
    \begin{equation*}
        G(x,y)=\frac{1}{2\pi}\log\frac{|\bar{x}-y|}{|x-y|},
    \end{equation*}
    where $\bar{x}=(x_1,-x_2)$ denotes reflection of $x=(x_1,x_2)$ in the $x_1$-axis. Define the Green's operator $\mathcal{G}$ as follows
    \begin{equation*}
        \mathcal{G}\zeta(x)=\int_\Pi G(x,y)\zeta(y)dy.
    \end{equation*}
    Let $J$ be the conjugate function to $I$ defined by \eqref{0-17}, and $J_{F}$ be defined as  
    \begin{equation*}
        J_{F}(s)=\sup _{t\in \R}\left[s t-F(t)\right]\  \text { if } s \geq 0 ;\  J_{F}(s)=0 \text { if } s<0,
    \end{equation*}
where $F(t)=\int_{0}^{t}f(s)ds$. 
    Then $J$ has the well-known properties
    \begin{equation}\label{1-g}
        J(x_2,s)=i(x_2,t)t -I(x_2,t), \ \ \partial_{ss}J(x_2,s)=[\partial_{t}i(x_2,t)]^{-1},
    \end{equation}
 with $ s=i(x_2,t) $ or, equivalently, $t=\partial_{s}J(x_2,s)$. Similarly, $J_{F}$ has the same property. Moreover,
      when $g=\alpha f$, it holds
    \begin{align*}
        J(x_2,s)&=\sup_{t\in \R}(st-I(x_2,t))=(1+\alpha x_2)\sup_{t\in \R}\Big(\frac{s}{1+\alpha x_2}t-F(t)\Big)\\ 
        &=(1+\alpha x_2)J_{F}\Big(\frac{s}{1+\alpha x_2}\Big).
    \end{align*}

    \subsection{Variational problem}
     In this paper, we consider this problem by using an improved vorticity method.
    
    For fixed $W>0$ and $\kappa>0$, we consider the energy functional given by
    \begin{equation*}
        \mathcal{E}_\varepsilon(\zeta)=\frac{1}{2}\int_\Pi{\zeta \mathcal{G}\zeta}dx-{W}\int_{\Pi}x_2\zeta dx-\frac{1}{\varepsilon^2}\int_\Pi J(x_2,\varepsilon^2\zeta)dx.
    \end{equation*}

  Let \[D=\{x\in \Pi~|~-1<x_1<1,\ r^*/2<x_2<2r^*\},  \] where $r^*>0$ is uniquely determined by the following
  \begin{equation}\label{2.1}
      -\frac{\kappa^2}{4\pi r^*}+\kappa W+\frac{\alpha}{1+\alpha r^*}\frac{\kappa^2}{8\pi}=0,
  \end{equation}  
    and
    \begin{equation*}
        \mathcal{A}_{\varepsilon,\Lambda}=\{\zeta\in L^\infty(\Pi)~|~ 0\le \zeta \le \frac{\Lambda}{\varepsilon^2}~ \text{a.e.}, \int_{\Pi}\zeta dx \le\kappa,~ supp(\zeta)\subseteq D \},
    \end{equation*}
    where $\varepsilon>0$ is a small parameter, $f(0^+)=\lim_{t\to 0^+}f(t)$ and $\Lambda>\max\{1,f(0^+)\}$ is an appropriate constant (whose value will be determined later, see Lemma \ref{le12} below). We will seek maximizers of $\mathcal{E}_\varepsilon$ relative to $\mathcal{A}_{\varepsilon,\Lambda}$.

    Let ${\zeta}^*$ be the Steiner symmetrization of $\zeta$ with respect to the line $x_1=0$ in $\Pi$ (see Appendix I of \cite{Bu0,Nor}).

    For the existence of maximizers for $\mathcal{E}_\varepsilon$ over $\mathcal{A}_{\varepsilon, \Lambda}$  we have

\begin{lemma}\label{le2.1}
    There exists $\zeta^{\varepsilon, \Lambda}=(\zeta^{\varepsilon, \Lambda})^* \in \mathcal{A}_{\varepsilon, \Lambda}$ such that
    \begin{equation*}
        \mathcal{E}_\varepsilon(\zeta^{\varepsilon, \Lambda})= \max_{\tilde{\zeta} \in \mathcal{A}_{\varepsilon, \Lambda}}\mathcal{E}_\varepsilon(\tilde{\zeta})<+\infty.
    \end{equation*}
  
\end{lemma}

\begin{proof} It is easy to verify that $\mathcal{E}_\epsilon$ is bounded from above over $\mathcal{A}_{\ep, \Lambda}$.
   We may take a sequence $\{\zeta_{k}\}\subset \mathcal{A}_{\varepsilon, \Lambda}$ such that as $k\to +\infty$
    \begin{equation*}
        \begin{split}
            \mathcal{E}_\varepsilon(\zeta_{k}) & \to \sup\{\mathcal{E}_\varepsilon(\tilde{\zeta})~|~\tilde{\zeta}\in \mathcal{A}_{\varepsilon, \Lambda}\}, \\
            \zeta_{k} & \to \zeta\in L^{\infty}({D})~~\text{weakly-star}.
        \end{split}
    \end{equation*}
    Clearly  $\zeta\in \mathcal{A}_{\varepsilon, \Lambda}$. Using the standard arguments (see \cite{Bu0}), we may assume that $\zeta_{k}=(\zeta_{k})^*$, and hence $\zeta=\zeta^*$. Since $G(\cdot,\cdot)\in L^1(D\times D)$, we have
    \begin{equation*}
        \lim_{k\to +\infty}\int_D{\zeta_{k} \mathcal{G}\zeta_{k}}dx = \int_D{\zeta \mathcal{G}\zeta}dx.
    \end{equation*}
    On the other hand, we have the lower semicontinuity for each  of the rest terms in $\mathcal{E}_{\varepsilon}$, namely,
    \begin{equation*}
        \begin{split}
            \liminf_{k\to +\infty} \int_{D}x_2\zeta_k dx  & \ge  \int_{D}x_2\zeta dx, \\
            \liminf_{k\to +\infty}\int_D J(x_2,\varepsilon^2\zeta_k)dx   & \ge \int_D J(x_2,\varepsilon^2\zeta)dx.
        \end{split}
    \end{equation*}
    Consequently, we conclude that $ \mathcal{E}_\varepsilon(\zeta)=\lim_{k\to +\infty} \mathcal{E}_\varepsilon(\zeta_k)=\sup_{\mathcal{A}_{\varepsilon, \Lambda}} \mathcal{E}_\varepsilon$, with $\zeta\in \mathcal{A}_{\varepsilon, \Lambda}$, which completes the proof.
\end{proof}

Since $\zeta^{\varepsilon, \Lambda} \in L^\infty(\Pi)$, it follows that $\mathcal{G}\zeta^{\varepsilon, \Lambda}\in C^1_{\text{loc}}(\Pi)$. Moreover, by $\zeta^{{\varepsilon, \Lambda}}=(\zeta^{\varepsilon, \Lambda})^*$ we can conclude that $\mathcal{G}\zeta^{\varepsilon, \Lambda}$ is symmetric decreasing in $x_1$ and $\partial_{x_1}\mathcal{G}\zeta^{\varepsilon, \Lambda}(x) <0$ if $\zeta^{\varepsilon, \Lambda}\not\equiv 0$ and $x_1>0$. Therefore, every level set of $\mathcal{G}\zeta^{{\varepsilon, \Lambda}}-{Wx_2}$ has measure zero by the implicit function theorem.

\begin{lemma}\label{le2.2}
    Let $\zeta^{{\varepsilon, \Lambda}}$ be a maximizer as in Lemma \ref{le2.1}, then there exists a Lagrange multiplier $\mu^{{\varepsilon, \Lambda}}\ge 0$ such that
    \begin{equation}\label{2-1}
        \zeta^{{\varepsilon, \Lambda}}=\frac{1}{ \varepsilon^2}i\left(x_2,\psi^{{\varepsilon, \Lambda}}\right)\chi_{A_{\varepsilon,\varLambda}}+\frac{\varLambda}{\varepsilon^2}\chi_{B_{\varepsilon,\varLambda}} \ \ a.e.\  \text{in}\  D,
    \end{equation}
    where
    \begin{equation}\label{2-2}
        \psi^{{\varepsilon, \Lambda}}=\mathcal{G}\zeta^{{\varepsilon, \Lambda}}-{Wx_2}-\mu^{{\varepsilon, \Lambda}},
    \end{equation}
and
\[A_{\varepsilon,\varLambda}:=\{(x_1,x_2)\in \Pi ~|~0<\psi^{{\varepsilon, \Lambda}}<\partial_{s}J(x_2,\varLambda)\},\]
 \[ B_{\varepsilon,\varLambda}:=\{(x_1,x_2)\in \Pi ~|~\psi^{{\varepsilon, \Lambda}} \ge \partial_{s}J(x_2,\varLambda)\}. \]
    Moreover, it holds $\int_D\zeta^{\varepsilon, \Lambda} dx=\kappa$ provided $\mathcal{E}_\varepsilon(\zeta^{\varepsilon, \Lambda})>0$  and every $\mu^{\varepsilon, \Lambda}$ satisfying \eqref{2-1} is positive.
\end{lemma}
 \begin{proof}
    Note that we may assume $\zeta^{\varepsilon,\Lambda} \not\equiv 0$, otherwise the conclusion is obtained by letting $\mu^{\varepsilon, \Lambda}=0$.
    We consider a family of variations of $\zeta^{\varepsilon, \Lambda}$
    \begin{equation*}
        \zeta_{(t)}=\zeta^{\varepsilon, \Lambda}+t(\tilde{\zeta}-\zeta^{\varepsilon, \Lambda}),\ \ \ t\in[0,1],
    \end{equation*}
    defined for arbitrary $\tilde{\zeta}\in \mathcal{A}_{\varepsilon, \Lambda}$. Since $\zeta^{\varepsilon, \Lambda}$ is a maximizer, we have
    \begin{equation*}
        \begin{split}
            0 & \ge \frac{d}{dt}\mathcal{E}_\varepsilon(\zeta_{(t)})|_{t=0^+} \\
            & =\int_{D}(\tilde{\zeta}-\zeta^{\varepsilon, \Lambda})\left[\mathcal{G}\zeta^{\varepsilon, \Lambda}-{Wx_2}-\partial_sJ(x_2,\varepsilon^2\zeta^{\varepsilon, \Lambda}) \right]dx.
        \end{split}
    \end{equation*}
    This implies that for any  $\tilde{\zeta}\in \mathcal{A}_{\varepsilon,\Lambda}$, there holds
    \begin{equation*}
        \begin{split}
            \int_{D}\zeta^{\varepsilon,\Lambda} &\left[\mathcal{G}\zeta^{\varepsilon, \Lambda}-{Wx_2}-\partial_sJ(x_2,\varepsilon^2\zeta^{\varepsilon, \Lambda})\right]dx  \\
            & \ \ \ \ \ \ \ \ \ \ \ \ \ \  \ge \int_{D}\tilde{\zeta}  \left[\mathcal{G}\zeta^{\varepsilon, \Lambda}-{Wx_2}-\partial_sJ(x_2,\varepsilon^2\zeta^{\varepsilon, \Lambda})\right]dx.
        \end{split}
    \end{equation*}
    By an adapted bathtub principle (see Lieb-Loss \cite{Lieb}, \S1.14), we obtain that for any point in $D$, it holds
    \begin{align}\label{2-3}
        \begin{cases}
            ~\mathcal{G}\zeta^{\varepsilon, \Lambda}-{Wx_2}- \mu^{\varepsilon, \Lambda} \ge &\partial_sJ(x_2,\varepsilon^2\zeta^{\varepsilon, \Lambda})\ \ \ \ \  \mbox{whenever}\  \zeta^{\varepsilon, \Lambda}=\frac{\Lambda}{\varepsilon^2}, \\
            ~\mathcal{G}\zeta^{\varepsilon, \Lambda}-{Wx_2}- \mu^{\varepsilon, \Lambda}= & \partial_sJ(x_2,\varepsilon^2\zeta^{\varepsilon, \Lambda})\ \ \ \ \  \mbox{whenever}\  0<\zeta^{\varepsilon, \Lambda}<\frac{\Lambda}{\varepsilon^2}, \\
            ~\mathcal{G}\zeta^{\varepsilon, \Lambda}-{Wx_2}- \mu^{\varepsilon, \Lambda}\le  & \partial_sJ(x_2,\varepsilon^2\zeta^{\varepsilon, \Lambda}) \ \ \ \ \  \mbox{whenever}\  \zeta^{\varepsilon, \Lambda}=0,
        \end{cases}
    \end{align}
    for some $\mu^{\varepsilon, \Lambda} \ge 0$. It follows that $\{0<\zeta^{\varepsilon, \Lambda}\le f(0^+)\varepsilon^{-2}\}\subseteq\{\mathcal{G}\zeta^{\varepsilon, \Lambda}-{Wx_2}=\mu^{\varepsilon, \Lambda}\}$. Since $-\Delta(\mathcal{G}\zeta^{\varepsilon, \Lambda}-{Wx_2})=\zeta^{\varepsilon, \Lambda}$ almost everywhere in $D$, we conclude that $\left|\{0<\zeta^{\varepsilon, \Lambda}\le f(0^+)\varepsilon^{-2}\}\right|=0$. Recall that for $t>0$ and $s>f(0^+)$, there holds
    \begin{equation*}
        t= \partial_sJ(x_2,s)\ \ \ \text{if and only if}\ \ \ s= i(x_2,t).
    \end{equation*}
    Now the stated form $\eqref{2-1}$ follows from \eqref{2-3} immediately. Moreover, when $\mathcal{E}_\varepsilon(\zeta^{\varepsilon, \Lambda})>0$ and $\mu^{\varepsilon, \Lambda}>0$, by the bathtub principle \cite{Lieb}, we have $\int_D\zeta^{\varepsilon,\Lambda} d\nu=\kappa$. The proof is thus complete.
\end{proof}

\subsection{Asymptotic behavior of the maximizers}
In the following, we study the asymptotic behavior of $\zeta^{\varepsilon, \Lambda}$ when $\varepsilon \to 0^+$. In the sequel, we will use the same $C$ to denote positive constants independent of $\varepsilon$.

To begin with, we give a lower bound of the energy.
\begin{lemma}\label{le2.3}
 We have
    \begin{equation*}
        \mathcal{E}_\varepsilon(\zeta^{{\varepsilon, \Lambda}})\ge \frac{\kappa^{2}}{4 \pi} \log \frac{1}{\varepsilon}-C.
    \end{equation*}
\end{lemma}

\begin{proof}
    The key idea is to select a suitable test function. Let $x_0\in D$ be a fixed point and $$\tilde{\zeta}_1^\varepsilon:=\frac{1}{\varepsilon^2} \chi_{_{B_{\varepsilon\sqrt{{\kappa}/{\pi}}}(x_0)}},$$
    where $\chi_{_A}$ denotes the characteristic function of a set $A$.
    It is clear that $\tilde{\zeta}_1^{\varepsilon} \in \mathcal{A}_{\varepsilon, \Lambda}$  if $\varepsilon$ is sufficiently small. By a simple calculation, we get
  \begin{equation*}
    \begin{split}
        \mathcal{E}(\tilde{\zeta}_1^{\varepsilon})\ge \frac{\kappa^{2}}{4 \pi} \log \frac{1}{\varepsilon}-C.
    \end{split}
\end{equation*}
    Since $\zeta^{\varepsilon, \Lambda}$ is a maximizer, we have $\mathcal{E}(\zeta^{\varepsilon, \Lambda})\ge \mathcal{E}(\tilde{\zeta}_1^{\varepsilon})$ and the proof is thus complete.
\end{proof}

   We now turn to estimate the Lagrange multiplier $\mu^{\varepsilon, \Lambda}$. 
\begin{lemma}\label{le2.4}
        There holds
        \begin{align*}
             \mu^{\varepsilon,\Lambda} & \ge \frac{2\mathcal{E}_\varepsilon(\zeta^{\varepsilon, \Lambda})}{\kappa} +\frac{W}{\kappa}\int_D x_2 \zeta^{\varepsilon, \Lambda}  dx-|2\delta_0-1|\kappa^{-1}\int_D\partial_sJ(x_2,\Lambda)\zeta^{\varepsilon,\Lambda}dx \\
            &\ \ \ \ \  -(C+o_{\Lambda\varepsilon^{-2}}(1))\left(1+\left(\int_D \partial_sJ(x_2,\Lambda)\zeta^{\varepsilon,\Lambda}dx\right)^\frac{1}{2}\right) .
        \end{align*}
\end{lemma}
\begin{proof}
       For convenience, let us abbreviate $(\zeta^{\varepsilon,\Lambda},\psi^{\varepsilon,\Lambda}, \mu^{\varepsilon, \Lambda})$ to $(\zeta, \psi,\mu)$ here. Recall that from $(H3)'$
       \begin{equation*}
        J(x_2,s)\ge (1-\delta_0)\partial_sJ(x_2,s)s-\delta_1s, \ \ \ \forall~t>0, \ \forall~0<x_2\le 2r^*.
    \end{equation*}
    By \eqref{2-2} and $(H3)'$, we then have
    \begin{align*}
         2\mathcal{E}_\varepsilon(\zeta) &=\int_D{\zeta \mathcal{G}\zeta}dx-{{2W}}\int_{D}x_2\zeta dx-\frac{2}{\varepsilon^2}\int_D J(x_2,\varepsilon^2\zeta)dx\\
        &= \int_D{\zeta \left(\mathcal{G}\zeta-{W}x_2-\mu\right)}dx-\frac{2}{\varepsilon^2}\int_D J(x_2,\varepsilon^2\zeta)dx\\
        &\ \ \ \ \ \ \ \ \ \ \ \ -{W}\int_{D}x_2 \zeta dx+\mu\int_D\zeta dx \\
        & \le \int_D \zeta \psi dx-2(1-\delta_0)\int_D\partial_sJ(x_2,\varepsilon^2\zeta)\zeta dx+2\delta_1\int_D\zeta dx     \\
        &\ \ \ \ \ \ \ \ \ \ \ \ -{{W}}\int_{D}x_2 \zeta dx+\mu\int_D\zeta dx .\\
    \end{align*}
    Thus we get
    \begin{equation}\label{2-4}
        2\mathcal{E}_\varepsilon(\zeta)\le \int_D \zeta \psi dx-2(1-\delta_0)\int_D\partial_sJ(x_2,\varepsilon^2\zeta)\zeta d\nu-{{W}}\int_{D}x_2 \zeta d\nu +\mu\kappa+2\delta_1\kappa.
    \end{equation}
     If we take $\psi_+ \in H_{0}^{1}(\Pi)$ as a test function, we then obtain
    \begin{equation}\label{2-5}
        \int_\Pi {|\nabla \psi_+|^2} dx =\int_D \zeta \psi dx.
    \end{equation}
    Let $U:=\left(\psi-\partial_sJ(x_2,\Lambda)\right)_+$. Recalling \eqref{2-3}, by H\"older's inequality and Sobolev's inequality, we have
    \begin{align*}
         \int_D \zeta \psi dx & = \int_D \zeta \left(\psi-\partial_sJ(x_2,\Lambda)\right)_+ dx+\int_D\zeta \partial_sJ(x_2,\varepsilon^2\zeta)dx \\
        & \le  \frac{\Lambda}{\varepsilon^2}|\{\zeta=\Lambda \varepsilon^{-2}\}|^\frac{1}{2} \left(\int_D U^2 dx\right)^\frac{1}{2}+\int_D\zeta \partial_sJ(x_2,\varepsilon^2\zeta)dx \\
        & \le  \frac{C\Lambda}{\varepsilon^2}|\{\zeta=\Lambda \varepsilon^{-2}\}|^\frac{1}{2} \left(\int_D U^2 dx\right)^\frac{1}{2}+\int_D\zeta \partial_sJ(x_2,\varepsilon^2\zeta)dx\\
        & \le  \frac{C\Lambda}{\varepsilon^2}|\{\zeta=\Lambda \varepsilon^{-2}\}|^\frac{1}{2} \left(\int_D |\nabla U|dx+\int_D |U|dx\right)+\int_D\zeta \partial_sJ(x_2,\varepsilon^2\zeta)dx\\
        & \le  \frac{C\Lambda}{\varepsilon^2}|\{\zeta=\Lambda \varepsilon^{-2}\}|^\frac{1}{2} \left(\int_{\{\zeta=\Lambda\varepsilon^{-2}\}}|\nabla \psi|dx+\int_{\{\zeta=\Lambda\varepsilon^{-2}\}}|\nabla \partial_sJ(x_2,\Lambda)|dx\right)\\
        &  \ \ \ \ \ \ \ \  \ \ \ \ \ \ \ \ \ \ +\frac{C\Lambda}{\varepsilon^2}|\{\zeta=\Lambda \varepsilon^{-2}\}|^\frac{1}{2}\int_{\{\zeta=\Lambda\varepsilon^{-2}\}} \psi dx +\int_D\zeta \partial_sJ(x_2,\varepsilon^2\zeta)dx\\
        & \le  \frac{C\Lambda}{\varepsilon^2}|\{\zeta=\Lambda \varepsilon^{-2}\}| \left(\int_D {|\nabla \psi|^2} dx\right)^\frac{1}{2}+o_{\Lambda\varepsilon^{-2}}(1)\int_D \zeta \psi dx+o_{\Lambda\varepsilon^{-2}}(1)+\int_D\zeta \partial_sJ(x_2,\varepsilon^2\zeta)dx\\
        & \le C \left(\int_D {|\nabla \psi_+|^2} dx\right)^\frac{1}{2}+o_{\Lambda\varepsilon^{-2}}(1)\int_D \zeta \psi dx+o_{\Lambda\varepsilon^{-2}}(1)+\int_D\zeta \partial_sJ(x_2,\varepsilon^2\zeta)dx.
    \end{align*}
     Combining \eqref{2-5} and above estimate, we get
    \begin{equation*}
        \int_D \zeta \psi dx\le C \left(\int_D \zeta \psi dx\right)^\frac{1}{2}+o_{\Lambda\varepsilon^{-2}}(1)\int_D \zeta \psi dx+o_{\Lambda\varepsilon^{-2}}(1)+\int_D\zeta \partial_sJ(x_2,\varepsilon^2\zeta)dx,
    \end{equation*}
    which implies
    \begin{equation}\label{2-6}
        \int_D \zeta \psi dx\le (C+o_{\Lambda\varepsilon^{-2}}(1))\left(1+\left(\int_D\zeta \partial_sJ(x_2,\varepsilon^2\zeta)dx\right)^\frac{1}{2}\right)+\int_D\zeta \partial_sJ(x_2,\varepsilon^2\zeta)dx.
    \end{equation}
Combining \eqref{2-4} and \eqref{2-6}, we obtain
\begin{equation*}
     \begin{split}
        2\mathcal{E}_\varepsilon(\zeta)\le & (C+o_{\Lambda\varepsilon^{-2}}(1))\left(1+\left(\int_D\zeta \partial_sJ(x_2,\varepsilon^2\zeta)dx\right)^\frac{1}{2}\right) +|2\delta_0-1|\int_D\partial_sJ(x_2,\varepsilon^2\zeta)\zeta dx \\
        &-{{W}}\int_{D}x_2 \zeta dx +\mu\kappa+2\delta_1\kappa,
    \end{split}
\end{equation*}
which clearly implies the desired result.
The proof is thus complete.
\end{proof}

 Combining Lemmas \ref{le2.3} and \ref{le2.4}, we immediately get the following estimate.
\begin{lemma}\label{le2.5}
   \begin{equation*}
       \begin{split}
           \mu^{\varepsilon,\Lambda} & \ge \frac{\kappa}{2\pi} \log{\frac{1}{\varepsilon}} +\frac{W}{\kappa}\int_Dx_2 \zeta^{\varepsilon, \Lambda}  dx-|2\delta_0-1|\kappa^{-1}\int_D\partial_sJ(x_2,\Lambda)\zeta^{\varepsilon,\Lambda}dx \\
           &\ \ \ \ \  -(C+o_{\Lambda\varepsilon^{-2}}(1))\left(1+\left(\int_D \partial_sJ(x_2,\Lambda)\zeta^{\varepsilon,\Lambda}dx\right)^\frac{1}{2}\right) .
       \end{split}
   \end{equation*}
\end{lemma}
It is easy to see that $\partial_sJ(x_2,\Lambda)\le \partial_sJ(2r^*,\Lambda)$ for $x_2\le 2r^*$. Hence $\int_D \partial_sJ(x_2,\Lambda)\zeta^{\varepsilon,\Lambda}dx\le C^*$. As a consequence of Lemmas \ref{le2.2}, \ref{le2.3} and \ref{le2.5}, we have

\begin{corollary}\label{le2.6}
    For each fixed $\Lambda$, there holds $\int_D\zeta^{\varepsilon, \Lambda} dx=\kappa$ when $\varepsilon$ is sufficiently small.
\end{corollary}

The following lemma shows that $\psi^{\varepsilon, \Lambda}$ has a priori upper bound with respect to $\Lambda$.
\begin{lemma}\label{le11}
    Let $\Lambda$ be fixed. Then for all sufficiently small $\varepsilon$, we have
    \begin{align*}
        \psi^{\varepsilon, \Lambda}(x) \le  &\ |2\delta_0-1|\kappa^{-1}\int_D\partial_sJ(x_2,\Lambda)\zeta^{\varepsilon,\Lambda}d\nu \\
        & +(C+o_{\Lambda\varepsilon^{-2}}(1))\left(1+\left(\int_D \partial_sJ(x_2,\Lambda)\zeta^{\varepsilon,\Lambda}d\nu\right)^\frac{1}{2}\right)+C\log \Lambda,\ \ \ \ a.e.\  \text{in}\  D.
    \end{align*}

\end{lemma}

\begin{proof}
    For any $x=(x_1,x_2)\in supp\,(\zeta^{\varepsilon,\Lambda})$, by Lemma \ref{le2.2} , we have 
    \begin{equation*}
        \begin{split}
              \psi^{\varepsilon,\Lambda}(x)& =\mathcal{G}\zeta^{\varepsilon,\Lambda}(x)-{Wx_2}-\mu^{\varepsilon,\Lambda} \\
            & = \int_D G(x,y)\zeta^{\varepsilon,\Lambda}(y)dy-Wx_2-\mu^{\varepsilon,\Lambda} \\
            & \le \frac{1}{2\pi}\int_D\log{\frac{1}{|x-y|}}\zeta^{\varepsilon,\Lambda}(y)dy-Wx_2-\mu^{\varepsilon,\Lambda}+C\\
            & \le  \frac{1}{2\pi} \frac{\varLambda}{\varepsilon^2}\int_{B_{\varepsilon\sqrt{\kappa/\varLambda\pi}}(x)} \log\frac{1}{|x-y|}\zeta^{\varepsilon,\varLambda}(y)dy -Wx_2-\mu^{\varepsilon,\Lambda}+C\\
            &\le \frac{\kappa}{2\pi}\log\frac{1}{\varepsilon}+\frac{\kappa}{4\pi}\log {\varLambda}-Wx_2-\mu^{\varepsilon,\Lambda}+C\\
            &\le \frac{\kappa}{2\pi}\log\frac{1}{\varepsilon}-\mu^{\varepsilon,\Lambda}+C(1+\log {\varLambda}),\\
        \end{split}
    \end{equation*}
    where the positive number $C$ does not depend on $\varepsilon$ and $\Lambda$.
    
 On the other hand, by Lemma \ref{le2.5} we have
    \begin{equation*}
        \begin{split}
            \mu^{\varepsilon,\Lambda} & \ge \frac{\kappa}{2\pi} \log{\frac{1}{\varepsilon}} +\frac{W}{\kappa}\int_Dx_2 \zeta^{\varepsilon, \Lambda}  dx-|2\delta_0-1|\kappa^{-1}\int_D\partial_sJ(x_2,\Lambda)\zeta^{\varepsilon,\Lambda}dx \\
            &\ \ \ \ \  -(C+o_{\Lambda\varepsilon^{-2}}(1))\left(1+\left(\int_D \partial_sJ(x_2,\Lambda)\zeta^{\varepsilon,\Lambda}dx\right)^\frac{1}{2}\right) .
        \end{split}
    \end{equation*}
    Thus we have
    \begin{equation*}
        \begin{split}
            \psi^{\varepsilon, \Lambda}(x) \le  &\ |2\delta_0-1|\kappa^{-1}\int_D\partial_sJ(x_2,\Lambda)\zeta^{\varepsilon,\Lambda}dx \\
            & +(C+o_{\Lambda\varepsilon^{-2}}(1))\left(1+\left(\int_D \partial_sJ(x_2,\Lambda)\zeta^{\varepsilon,\Lambda}dx\right)^\frac{1}{2}\right)+C\log \Lambda,
        \end{split}
    \end{equation*}
    and the proof is thus complete.
\end{proof}

Now we can choose a suitable $\Lambda$ such that the patch part in \eqref{2-1} will vanish when $\varepsilon$ is suffciently small.
\begin{lemma}\label{le12}
    There exists $\Lambda_0>\max\{1,f(0^+)\}$ such that
    \begin{equation*}
        \zeta^{\varepsilon, \Lambda_0}=\frac{1}{ \varepsilon^2}i(x_2,\psi^{\varepsilon, \Lambda_0}),\ \ \ \ \ a.e.\  \text{in}\  D,
    \end{equation*}
    provided $\varepsilon>0$ is sufficiently small.
\end{lemma}

\begin{proof}
    Notice that on the patch part $\{\zeta^{\varepsilon, \Lambda}={\Lambda\varepsilon^{-2}}\}$, we have
    \begin{equation}\label{2-8}
        \psi^{\varepsilon, \Lambda}(x_1,x_2)\ge \partial_s J(x_2,\Lambda)= \partial_s J(r^*,\Lambda)+o_{\Lambda\varepsilon^{-2}}(1).
    \end{equation}
    On the other hand, by Lemma \ref{le11}, we have
    \begin{equation}\label{2-9}
        \begin{split}
            \psi^{\varepsilon,\Lambda}& \le |2\delta_0-1|\kappa^{-1}\int_D\partial_sJ(x_2,\Lambda)\zeta^{\varepsilon,\Lambda}dx\\
            & \ \ \ \ \ \ +(C+o_{\Lambda\varepsilon^{-2}}(1))\left(1+\left(\int_D \partial_sJ(x_2,\Lambda)\zeta^{\varepsilon,\Lambda}dx\right)^\frac{1}{2}\right)+C\log \Lambda \\
            & \le |2\delta_0-1|\partial_s J(r^*,\Lambda) +(C+o_{\Lambda\varepsilon^{-2}}(1))\left(1+\left(\partial_sJ(r^*,\Lambda)\right)^\frac{1}{2}\right)+C\log\Lambda, \ \ a.e.\  \text{in}\  D.
        \end{split}
    \end{equation}
    By assumption $(H4)$, we see that for all $\tau>0$
    \begin{equation*}
        \lim_{s\to+\infty}\left(\partial_s J(r^*,s)-\tau\log s \right)=+\infty.
    \end{equation*}
    Hence we can choose $\Lambda_0>\max\{1,f(0^+)\}$ such that
    \begin{equation}\label{2-10}
        (1-|2\delta_0-1|)\partial_s J(r^*,\Lambda_0)-(C+1)\left(1+\left(\partial_sJ(r^*,\Lambda_0)\right)^\frac{1}{2}\right)-C\log\Lambda_0>1.
    \end{equation}
    From \eqref{2-8},\eqref{2-9} and \eqref{2-10}, we conclude that $| \{\zeta^{\varepsilon, \Lambda_0}={\Lambda_0\varepsilon^{-2}}\}|=0$ when $\varepsilon$ is sufficiently small, which completes the proof.
\end{proof}

For simplicity, we will abbreviate $(\zeta^{\varepsilon,\Lambda_{0}},\psi^{\varepsilon,\Lambda_{0}}, \mu^{\varepsilon, \Lambda_{0}})$ to $(\zeta^{\varepsilon}, \psi^{\varepsilon},\mu^{\varepsilon})$ in the sequel.

\begin{lemma}\label{lem2-9}
    The following asymptotic expansions hold as $\varepsilon\to 0^+$:
    \begin{align}
        \label{2-11}  \mathcal{E}_\varepsilon(\zeta^\varepsilon) & =\frac{\kappa^2}{4\pi}\ln\frac{1}{\varepsilon}+O(1), \\
        \label{2-12}  \mu^\varepsilon & =\frac{\kappa}{2\pi}\ln\frac{1}{\varepsilon}+O(1).
    \end{align}
\end{lemma}

\begin{proof}
    Using Riesz's rearrangement inequality, we obtain
    \begin{equation*}
        \mathcal{E}_\varepsilon(\zeta^\varepsilon)\le \frac{\kappa^2}{4\pi}\ln\frac{1}{\varepsilon}+C.
    \end{equation*}
    Combining this and Lemma \ref{le2.3}, we get \eqref{2-11}. Using the same argument as in the proof of Lemma \ref{le2.5}, we can easily get \eqref{2-12}. The proof is complete.
\end{proof}

Now we turn to estimate the size of the supports of $\zeta^\varepsilon$ for $\varepsilon$ small. To
this end, we first recall an auxiliary lemma.
\begin{lemma}[\cite{Cao2}, Lemma 2.8]\label{le2.7}
    Let $\Omega\subset D$, $0<\varepsilon<1$, $\eta\ge0$, and let $\xi$ be a non-negative function satisfying $\xi\in L^1(D)$, $\int_D \xi(x)dx=1$ and $||\xi||_{L^p(D)}\le C_1 \varepsilon^{-2(1-{1}/{p})}$ for some $1<p\le +\infty$ and $C_1>0$. Suppose for any $x\in \Omega$, there holds
    \begin{equation}\label{2-7}
        (1-\eta)\log\frac{1}{\varepsilon}\le \int_D \log\frac{1}{|x-y|}\xi(y)dy+C_2,
    \end{equation}
    where $C_2$ is a positive constant.
    Then there exists some constant $R>1$ such that
    \begin{equation*}
        diam(\Omega)\le R\varepsilon^{1-2\eta}.
    \end{equation*}
    The constant $R$ may depend on $C_1$, $C_2$, but not on $\eta$, $\varepsilon$.
\end{lemma}

Note that
\begin{equation*}
    G(x,y)\le \frac{1}{2\pi}\log\frac{1}{|x-y|}+C\ \ \  \text{in}\ D\times D,
\end{equation*}
and
\begin{equation*}
    \text{supp}(\zeta^\varepsilon)\subseteq \{x\in D\mid \mathcal{G}\zeta^\varepsilon(x)-Wx_2 \ge \mu^\varepsilon\}.
\end{equation*}

Using Lemmas \ref{le2.5} and \ref{le2.7}, we immediately get the following estimate.
\begin{lemma}\label{le2.8}
    Let $\zeta^{\varepsilon}$ be a maximizer as in Lemma \ref{le2.1}, then for $\varepsilon$ small it holds $\text{supp}(\zeta^\varepsilon)\le C\varepsilon$.
\end{lemma}

We proceed to determine the limiting location and the asymptotic shape of $\operatorname{supp}\left(\zeta^{\varepsilon}\right)$. Define the center of $\zeta^{\varepsilon}$ by
$$
x^{\varepsilon}:=\frac{1}{\kappa} \int_{\Pi} x \zeta^{\varepsilon}(x) d x .
$$
From now on for the remainder of the discussion we fix a sequence $\varepsilon=\varepsilon_{j} \rightarrow 0^{+}$such that

\begin{equation}\tag{$\ast$}\label{star}
    x^{\varepsilon} \rightarrow (0,\bar{r}) , \text { as } \varepsilon=\varepsilon_{j} \rightarrow 0^{+}, \text{where} \ \bar{r} \in [r^*/2,2r^*]. 
\end{equation}

Next we study the limiting behavior of the corresponding stream functions $\psi^{\varepsilon}$. Define the scaled versions of $\psi^{\varepsilon}$ as follows
$$
 \Psi^{\varepsilon}(y):=\psi^{\varepsilon}(x^{\varepsilon}+\varepsilon y),\ \ \ \ y \in \Pi^{\varepsilon}:=\left\{y \in \mathbb{R}^{2} \mid x^{\varepsilon}+\varepsilon y \in \Pi\right\}.
$$
Thus, we have
\begin{equation}\label{2-13}
   -\Delta \Psi^{\varepsilon}=(1+\alpha (x_2^{\varepsilon}+\varepsilon y_2))f(\Psi^{\varepsilon})\ \ \text { in } \Pi^{\varepsilon},
\end{equation}
\begin{equation}\label{2-14}
   \int_{\Pi^{\varepsilon}}(1+\alpha (x_2^{\varepsilon}+\varepsilon y_2))f(\Psi^{\varepsilon}) d y=\kappa.
\end{equation}

Note that $\left\{x \in D \mid \Psi^{\varepsilon}(x)>0\right\} \subset B_{R_{0}}(0)$. As mentioned in section \ref{sect 1}, we introduce the limiting profile $U^{\kappa}: \mathbb{R}^{2} \rightarrow \mathbb{R}$ defined as the unique radially symmetric solution of the problem
\begin{equation}
    \left\{\begin{array}{l}
        -\Delta U^{\kappa}=f(U^{\kappa}), \quad x \in \mathbb{R}^{2} \\
        \int_{\mathbb{R}^{2}}f(U^{\kappa}) d x=\kappa.
    \end{array}\right.
\end{equation}

\begin{lemma}\label{lem2.13}
  As $\varepsilon \rightarrow 0^{+}$, we have $\Psi^{\varepsilon}(x) \rightarrow U^{\kappa}\left(\left(1+\alpha \bar{r}\right)^{\frac{1}{2}} x\right)$ in $C_{\text {loc }}^{1, \gamma}\left(\mathbb{R}^{2}\right)$, where $\bar{r}$ is the same as in \eqref{star}
\end{lemma}
\begin{proof}
     Note that $\left\{\varepsilon^2 \zeta^{\varepsilon}(x^{\varepsilon}+\varepsilon y)\right\}$ is bounded in $L^{\infty}\left(\Pi^{\varepsilon}\right)$. Thus, by classical elliptic estimates, the sequence $\left\{\Psi^{\varepsilon}\right\}$ is bounded in $W_{\text {loc }}^{2, p}\left(\Pi^{\varepsilon}\right)$ for every $1 \leq p<+\infty$. By the Sobolev embedding theorem, we may conclude $\left\{\Psi^{\varepsilon}\right\}$ is compact in $C_{\text {loc }}^{1, \gamma}\left(\Pi^{\varepsilon}\right)$ for every $0<\gamma<1$. Up to a subsequence we may assume $\varepsilon^2 \zeta^{\varepsilon}(x^{\varepsilon}+\varepsilon y) \rightarrow \zeta(y)$ weakly-star in $L^{\infty}\left(\R^2\right)$ and $\Psi^{\varepsilon} \rightarrow \Psi$ in $C_{\text {loc }}^{1, \gamma}\left(\R^2\right)$. Let $\tilde{\Psi}(x):=\Psi\left(x /\left(1+\alpha \bar{r} \right)^{\frac{1}{2}}\right)$. By virtue of \eqref{2-13},\eqref{2-14} and $f$ is non-negative, we get $$
     -\Delta \tilde{\Psi}=f(\tilde{\Psi}) \text { in } \mathbb{R}^{2}, \quad \int_{\mathbb{R}^{2}} f(\tilde{\Psi}) d x=\kappa.
     $$ In view of Lemma \ref{lem2-9}, we know that $\zeta$ is a radially non-increasing function. Using this fact, \eqref{2-12} and Green's representation, it is not hard to find that $\Psi$ and $\tilde{\Psi}$ are also radial. Therefore, we have $\tilde{\Psi} \equiv U^{\kappa}$. Thus the proof is complete.
\end{proof}
\begin{remark}
    With the above results in hand, one can further study the shape of the free boundary $\{x\in D \mid \psi^\varepsilon(x)=0\}$ by the standard scaling techniques (see \cite{Cao2, Tur83}). The vortex core will be approximately a disk.
\end{remark}

As a consequence of Lemma \ref{lem2.13}, we can obtain as follows.
\begin{corollary}\label{lem2.15}
   As $\varepsilon \rightarrow 0^{+}$, one has $\varepsilon^2\zeta^{\varepsilon}(x^{\varepsilon}+\varepsilon x) \rightarrow\left(1+\alpha \bar{r}\right)f\Big(U^{\kappa}\big(\left(1+\alpha \bar{r}\right)^{\frac{1}{2}} x\big)\Big)$ in $L^{\infty}\left(\mathbb{R}^{2}\right)$.
\end{corollary}

We have shown that the vorticies $\{\zeta^\varepsilon\}$ would shrink to some point in $\bar{D}$ when $\varepsilon \to 0$. Now we investigate this limiting location. Let
\begin{equation}\label{2-18}
    \mathcal{W}(t)=\frac{\kappa^2}{4\pi}\log\frac{1}{2t}+{\kappa Wt}+\int_{\R^2}\frac{1+\alpha r^{*}}{1+\alpha t} J_{F}\left(\frac{1+\alpha t}{1+\alpha r^{*}}f(U^{\kappa}(z))\right)dz,\ \ t>0.
\end{equation}
  We want to show  that $r^*$ is the unique minimum point  $\mathcal{W}(r^{*})=\min_{t>0} \mathcal{W}(t)$. By calculating
  \begin{equation*}
\begin{split}
    \frac{d \mathcal{W}(t)}{dt}&=-\frac{\kappa^2}{4\pi t}+\kappa W-\int_{\R^2}{\alpha}\frac{1+\alpha r^*}{(1+\alpha t)^2}J_{F}\Big(\frac{1+\alpha t}{1+\alpha r^*}f\big(U^{\kappa}(z)\big)\Big)dz\\
    &\ \ \ \ \ \ \ +\frac{\alpha}{1+\alpha t}\int_{\R^2}\partial_{s}J_{F}\Big(\frac{1+\alpha t}{1+\alpha r^*}  f\big(U^{\kappa}(z)\big)\Big) f\big(U^{\kappa}(z)\big)dz.\\
\end{split}
\end{equation*}
Let 
\begin{equation*}
     K(s):=\partial_{s}J_{F}(s)s-J_{F}(s).
\end{equation*}
By \eqref{1-g}, we note that 
\begin{equation*}
    K(s)=J_{F}^*(\sigma)=F(\sigma), \ \text{for} \ s>0,
\end{equation*}
where $J_{F}^*$ be the conjugate function to $J_{F}$ and $\sigma=\partial_{s}J_{F}(s)$ is non-decreasing with respect to $s$, then we get $K$ is non-decreasing. Recall  that $r^{*}$ defined by \eqref{2.1}.
Thus we have
\begin{align*}
    \frac{d \mathcal{W}(t)}{dt}=\frac{1}{{(1+\alpha t)^2}}\Big(&M(t)-M(r^*)+{\alpha}{(1+\alpha r^*)}\times\\
    & \int_{\R^2}\Big(K\Big(\frac{1+\alpha t}{1+\alpha r^*}f\big(U^{\kappa}(z)\big)\Big)-K(f\big(U^{\kappa}(z)\big))\Big)dz\Big),
\end{align*}
where 
\begin{equation*}
    M(t)=(1+\alpha t)^2(-\frac{\kappa^2}{4 \pi t}+\kappa W).
\end{equation*}
We notice that $ W\ge \frac{(2-\alpha)_{+}^2\kappa}{32 \pi}$, so
\begin{equation*}
    M'(t)=\kappa\frac{(1+\alpha t)}{(4\pi t^2)}(8\pi W t^2-(2-\alpha)\kappa t+\kappa)\ge 0.\
\end{equation*}
$M'(t)$ is zero at most finitely many points, so $r^*$ is the only minimum point. Hence we obtain $\mathcal{W}(r^{*})=\min_{t>0} \mathcal{W}(t)$.

The following lemma show that the limiting location is $b_1$.
\begin{lemma}\label{le2.9}
    One has
    \begin{equation*}
        dist\left(supp(\zeta^\varepsilon), b_1\right)\to 0, \ \ \text{as}\ \varepsilon \to 0
    \end{equation*}
\end{lemma}
\begin{proof}
    Take $(0,x^\varepsilon_2)\in \text{supp}(\zeta^\varepsilon)$. Suppose (up to a subsequence) $x_2^\varepsilon\to \bar{r}>0$, as $\varepsilon \to 0$. We now show that $\bar{r}=r^{*}$. Set
    $\tilde{\zeta}^\varepsilon=\zeta^\varepsilon(\cdot-r^{*}\mathbf{e}_2+\bar{r}										\mathbf{e}_2)\in \mathcal{A}_{\varepsilon, \Lambda_0}$. Then
    \begin{equation*}
        \int_D \int_D{\tilde{\zeta}^\varepsilon(x)\log\frac{1}{|x-x'|}\tilde{\zeta}^\varepsilon}(x')dxdx'= \int_D \int_D{{\zeta}^\varepsilon(x)\log\frac{1}{|x-x'|}{\zeta}^\varepsilon}(x')dxdx'.
    \end{equation*}
   By Corollary \ref{lem2.15} , we have
    \begin{align*}
        \frac{1}{\varepsilon^2}\int_{D}J(x_2,\varepsilon^2 \zeta^{\varepsilon})dx    &=\frac{1}{\varepsilon^2}\int_{D}(1+\alpha x_2)J_{F}\Big(\frac{\varepsilon^2 \zeta^{\varepsilon}}{1+\alpha x_2}\Big)dx\\
        &=\int_{\Pi^{\varepsilon}}(1+\alpha x_2^{\varepsilon}+\varepsilon \alpha y_2)J_{F}\Big(\frac{\varepsilon^2 \zeta^{\varepsilon}(x^{\varepsilon}+\varepsilon y)}{1+\alpha x_2^{\varepsilon}+\varepsilon \alpha y_2}\Big)dy\\
        &=\int_{\R^2}J_{F}(f(U^{\kappa}(z)))dz+o(1).\\
    \end{align*}
Similarly, we can obtain
\begin{equation*}
     \frac{1}{\varepsilon^2}\int_{D}J(x_2,\varepsilon^2 \tilde{\zeta}^{\varepsilon})dx=\int_{\R^2}\frac{1+\alpha r^{*}}{1+\alpha \bar{r}} J_{F}\left(\frac{1+\alpha \bar{r}}{1+\alpha r^{*}}f(U^{\kappa}(z))\right)dz+o(1).\\
\end{equation*}
    Since $\mathcal{E}_\varepsilon(\tilde{\zeta}^\varepsilon)\le \mathcal{E}_\varepsilon(\zeta^\varepsilon)$, it follows that
    \begin{equation*}
        \begin{split}
            &\frac{1}{4\pi} \int_D \int_D{{\zeta}^\varepsilon(x)\log\frac{1}{|\bar{x}-x'|}{\zeta}^\varepsilon}(x')dxdx'+W\int_{D}x_2{\zeta}^\varepsilon dx +\frac{1}{\varepsilon^2}\int_{D}J(x_2,\varepsilon^2 \zeta^{\varepsilon})dx  \\
           &\ \ \ \  \le \frac{1}{4\pi} \int_D \int_D{\tilde{\zeta}^\varepsilon(x)\log\frac{1}{|\bar{x}-x'|}\tilde{\zeta}^\varepsilon}(x')dxdx'+W\int_{D}x_2\tilde{\zeta}^\varepsilon dx+\frac{1}{\varepsilon^2}\int_{D}J(x_2,\varepsilon^2 \tilde{\zeta^{\varepsilon}})dx.
        \end{split}
    \end{equation*}
    Letting $\varepsilon \to 0$, we get
    \begin{equation*}
        \mathcal{W}(\bar{r})\le \mathcal{W}(r^{*}).
    \end{equation*}
    This implies $\bar{r}=r^{*}$ and the proof is thus complete.
\end{proof}

Combining Lemma \ref{le2.8} and Lemma \ref{le2.9}, we get
\begin{lemma}\label{le2.10}
    For all sufficiently small $\varepsilon$, it holds $dist\left(supp(\zeta^\varepsilon), \partial D\right)>0$.
\end{lemma}

\begin{lemma}\label{}
    For all sufficiently small $\varepsilon$, one has
   \begin{equation*}
    \zeta^{\varepsilon, \Lambda_0}=\frac{1}{ \varepsilon^2}i(x_2,\psi^{\varepsilon, \Lambda_0}),\ \ \ \ \ a.e.\  \text{in}\  \Pi.
\end{equation*}
\end{lemma}

\begin{proof}
    By Lemma \ref{le12}, it suffices to show that $\psi^{\varepsilon}\leq 0$ a.e. on $\Pi\backslash D$.  By Lemma \ref{le2.10}, we have
    $$
    \mathcal{G} \zeta^{\varepsilon}(x) \leq C, \quad \forall x \in \Pi \backslash D
    $$
    when $\varepsilon$ is small enough.  In view of Lemma \ref{le2.5}, it follows that for $\varepsilon$  small, we have $\psi^{\varepsilon} \leq 0$ a.e. on $\Pi \backslash D$ and the proof is thus complete. 
\end{proof}

We  now turn to prove Theorem \ref{thm1}.
\begin{proof}[Proof of Theorem \ref{thm1}]
    It follows from the above lemmas by letting $\lambda=1/\varepsilon^{2}$.
\end{proof}

\section{Proofs of Theorems \ref{thm2} and \ref{thm3}}\label{sect3}
In order to complete the proof of Theorems \ref{thm2} and \ref{thm3}, we only need to re-determine the limit position where the vortex core will converge. Before proving Theorem \ref{thm2}, we need some preliminary lemmas.

We first study the limiting behavior of the corresponding stream functions $\psi^{\varepsilon}$. Define the scaling form of the stream function as follows
$$
\Psi_1^{\varepsilon}(y):=\psi^{\varepsilon}(x^{\varepsilon}+\varepsilon y),\ \ \ \ y \in \Pi^{\varepsilon}:=\left\{y \in \mathbb{R}^{2} \mid x^{\varepsilon}+\varepsilon y \in \Pi\right\}.
$$
By the properties of the stream function, we have
\begin{equation}
    -\Delta \Psi_1^{\varepsilon}=f(\Psi_1^{\varepsilon})+(x_2+\varepsilon y_2)\delta_{\varepsilon} g(\Psi_1^{\varepsilon})\ \ \text { in } \Pi^{\varepsilon},
\end{equation}
\begin{equation}
\int_{\Pi^{\varepsilon}}f(\Psi_1^{\varepsilon})+(x_2+\varepsilon y_2)\delta_{\varepsilon} g(\Psi_1^{\varepsilon}) d y=\kappa.
\end{equation}

Similar to Lemma \ref{lem2.13} for the limiting behavior of $\Psi_1^{\varepsilon}$, we have 
\begin{lemma}
      As $\varepsilon \rightarrow 0^{+}$, $\Psi_1^{\varepsilon}(x) \rightarrow U^{\kappa}\left( x\right)$ in $C_{\text {loc }}^{1, \gamma}\left(\mathbb{R}^{2}\right).$
\end{lemma}
Since the proof process is similar, we omit it here.

From the above lemma, we have the following
\begin{corollary}\label{lem3-2}
     As $\varepsilon \rightarrow 0^{+}$, $\varepsilon^2\zeta^{\varepsilon}(x^{\varepsilon}+\varepsilon x) \rightarrow f(U^{\kappa}(x))$ in $L^{\infty}\left(\mathbb{R}^{2}\right)$.
\end{corollary}
By Lemma \ref{le2.8}, we obtain that the vortex will converge to some point $(0,\bar{r})$ in $\bar{D}$. Now let's figure out where the limit point should be. Let 
\begin{equation}\label{3-3}
    \mathcal{W}_1(t)=\frac{\kappa^2}{4\pi}\log\frac{1}{2t}+{\kappa Wt}.
\end{equation}
To prove Theorem \ref{thm2}, instead of using \eqref{2.1} as the definition of $r^*$, we choose $r_1^*=\frac{\kappa}{4\pi W}$. It's easy to see that $\mathcal{W}_1(r_1^{*})=\min_{t>0} \mathcal{W}_1(t)$.

The following lemma shows that the limiting position is indeed $b_1=r_1^{*} \mathbf{e}_2$.
\begin{lemma}\label{le3.3}
    One has
    \begin{equation*}
        dist\left(supp(\zeta^\varepsilon), b_1\right)\to 0, \ \ \text{as}\ \varepsilon \to 0
    \end{equation*}
\end{lemma}
\begin{proof}
    Take $(0,x^\varepsilon_2)\in \text{supp}(\zeta^\varepsilon)$. Suppose (up to a subsequence) that $x_2^\varepsilon\to \bar{r}>0$, as $\varepsilon \to 0$. We now show that $\bar{r}=r_1^{*}$. Set
    $\tilde{\zeta}^\varepsilon=\zeta^\varepsilon(\cdot-r_1^{*}\mathbf{e}_2+\bar{r}\mathbf{e}_2)\in \mathcal{A}_{\varepsilon, \Lambda_0}$. Then
    \begin{equation*}
        \int_D \int_D{\tilde{\zeta}^\varepsilon(x)\log\frac{1}{|x-x'|}\tilde{\zeta}^\varepsilon}(x')dxdx'= \int_D \int_D{{\zeta}^\varepsilon(x)\log\frac{1}{|x-x'|}{\zeta}^\varepsilon}(x')dxdx'.
    \end{equation*}
    By Corollary \ref{lem3-2} , we have
    \begin{equation*}
        \begin{split}          
        \frac{1}{\varepsilon^2}\int_{D}J(x_2,\varepsilon^2 \zeta^{\varepsilon} )dx&=\int_{\Pi^{\varepsilon}}J(x_2^{\varepsilon}+\varepsilon y_2,\varepsilon^2 \zeta^{\varepsilon}(x^\varepsilon+\varepsilon y)) dy\\
        &=\int_{\Pi^{\varepsilon}}\sup_{t}\Big(\varepsilon^2 \zeta^{\varepsilon}(x^{\varepsilon}+\varepsilon y)t-\Big[\int_{0}^{t}f(s)+(x_2^{\varepsilon}+\varepsilon y_2)\delta_{\varepsilon} g(s)ds\Big]\Big)dy \\
        &=\int_{\R^2}\sup_{t}(f(U^{\kappa})t-F(t))dy+o(1)\\
        &=\int_{\R^2}J_{F}(f(U^{\kappa}))dy+o(1).\\
        \end{split}
    \end{equation*}
    Similarly, we can obtain
    \begin{equation*}
        \frac{1}{\varepsilon^2}\int_{D}J(x_2,\varepsilon^2 \tilde{\zeta}^{\varepsilon}(x) )dx=\int_{\R^2}J_{F}(f(U^{\kappa}))dy+o(1).
    \end{equation*}
    Since $\mathcal{E}_\varepsilon(\tilde{\zeta}^\varepsilon)\le \mathcal{E}_\varepsilon(\zeta^\varepsilon)$, it follows that
    \begin{equation*}
        \begin{split}
            &\frac{1}{4\pi} \int_D \int_D{{\zeta}^\varepsilon(x)\log\frac{1}{|\bar{x}-x'|}{\zeta}^\varepsilon}(x')dxdx'+W\int_{D}x_2{\zeta}^\varepsilon dx +\frac{1}{\varepsilon^2}\int_{D}J(x_2,\varepsilon^2 \zeta^{\varepsilon})dx  \\
            &\ \ \ \  \le \frac{1}{4\pi} \int_D \int_D{\tilde{\zeta}^\varepsilon(x)\log\frac{1}{|\bar{x}-x'|}\tilde{\zeta}^\varepsilon}(x')dxdx'+W\int_{D}x_2\tilde{\zeta}^\varepsilon dx+\frac{1}{\varepsilon^2}\int_{D}J(x_2,\varepsilon^2 \tilde{\zeta^{\varepsilon}})dx.
        \end{split}
    \end{equation*}
    Letting $\varepsilon \to 0$, we get
    \begin{equation*}
        \mathcal{W}_1(\bar{r})\le \mathcal{W}_1(r_1^{*}).
    \end{equation*}
    This implies $\bar{r}=r_1^{*}$ and the proof is thus complete.
\end{proof}

\begin{proof}[Proof of Theorem \ref{thm2}]
   Combining Lemma \ref{le2.1} to Lemma \ref{le2.8} in section \ref{sect2} and Lemma \ref{le3.3} in section \ref{sect3}, we prove that Theorem 1.4 is true by taking $\lambda=\frac{1}{\varepsilon^2}$.
\end{proof}

Finally we turn to proving Theorem \ref{thm3}. We modify the area slightly as follows
\[D=\{x\in \Pi~|~-1<x_1<1,\ c_1r_2^*< x_2< r_2^*/c_1\},  \]
where $c_1\in (0,1)$ will be given later by \eqref{C_choose}, and $r_2^*$ will be given by \eqref{3-6}.

 Similar to the previous process, we first study the limiting behavior of the corresponding stream function $\psi^{\varepsilon}$. We define the scaling form of the stream function below
$$
\Psi_2^{\varepsilon}(y):=\psi^{\varepsilon}(x^{\varepsilon}+\varepsilon y),\ \ \ \ y \in \Pi^{\varepsilon}:=\left\{y \in \mathbb{R}^{2} \mid x^{\varepsilon}+\varepsilon y \in \Pi\right\},
$$
where $\Psi_2$ satisfies
  \begin{equation}
    -\Delta \Psi_2^{\varepsilon}=\delta_{\varepsilon}f(\Psi_2^{\varepsilon})+(x_2^{\varepsilon}+\varepsilon y_2)g(\Psi_2^{\varepsilon})\ \ \text { in } \Pi^{\varepsilon},
\end{equation}
\begin{equation}
    \int_{\Pi^{\varepsilon}}\delta_{\varepsilon}f(\Psi_2^{\varepsilon})+(x_2^{\varepsilon}+\varepsilon y_2)g(\Psi_2^{\varepsilon}) d y=\kappa.
\end{equation}
Let's introduce the limiting profile $V^{\kappa}: \mathbb{R}^{2} \rightarrow \mathbb{R}$ defined as the unique radially symmetric solution of the problem
 \begin{equation}
    \left\{\begin{array}{l}
        -\Delta V^{\kappa}=g(V^{\kappa}), \quad x \in \mathbb{R}^{2} \\
        \int_{\mathbb{R}^{2}}g(V^{\kappa}) d x=\kappa.
    \end{array}\right.
\end{equation}
Thus, the asymptotic behavior of $\Psi_2^{\varepsilon}$ is described as follows
\begin{lemma}
     As $\varepsilon \rightarrow 0^{+}$, we have $\Psi_2^{\varepsilon}(x) \rightarrow V^{\kappa}\left( \sqrt{\bar{r}}x\right)$ in $C_{\text {loc }}^{1, \gamma}\left(\mathbb{R}^{2}\right).$
\end{lemma}
As a corollary to the above theorem, we have the following result
\begin{corollary}\label{lem3.5}
     As $\varepsilon \rightarrow 0^{+}$, one has $\varepsilon^2\zeta^{\varepsilon}(x^{\varepsilon}+\varepsilon x) \rightarrow \bar{r}g(V
     ^{\kappa}(\sqrt{\bar{r}}x))$ in $L^{\infty}\left(\mathbb{R}^{2}\right)$.
\end{corollary}
According to Lemma  \ref{le2.8}, we know that the vortex  will converge to some point $(0,\bar{r})$  in $\bar{D}$, where $\bar{r}\in [c_1 r_2^*,r_2^*/c_1]$.
We want to find the position of the limit of vortex convergence, so we define the following function
\begin{equation}\label{3-7}
    \mathcal{W}_{2}(t)=\frac{\kappa^2}{4\pi}\log\frac{1}{2t}+{\kappa Wt}+  \frac{r_2^*}{t}\int_{\R^2}J_{G}\Big(\frac{t}{r_2^*}g(V^{\kappa}(z))\Big)dz.
\end{equation}
Here we define $r_2^*>0$ as follows
\begin{equation}\label{3-6}
    r_2^*=\frac{1}{\kappa W}\Big(\frac{\kappa^2}{4\pi}-\mathcal{C}_{g, \kappa}\Big)=\frac{\kappa}{8\pi  W},
\end{equation}
with
\begin{equation}\label{3-9}
    \begin{split}
         \mathcal{C}_{g, \kappa}&=\int_{\R^2}\Big(V^{\kappa}(z)g(V^{\kappa}(z))-J_{G}\Big(g(V^{\kappa}(z))\Big)\Big)dz\\
         &=\int_{\R^2}G(V^{\kappa}(z))dz=\frac{\kappa^2}{8\pi}.
    \end{split}
\end{equation}
 We want to get $\mathcal{W}_2(r_2^{*})=\min_{t\in [c_1 r_2^*,r_2^*/c_1]}\mathcal{W}_2(t)$. So by taking the derivative of $\mathcal{W}_2$, we have
 \begin{equation*}
     \begin{split}
          \frac{d \mathcal{W}_2(t)}{dt}&=-\frac{\kappa^2}{4\pi t}+\kappa W-\frac{r_2^*}{t^2}\int_{\R^2}J_{G}\Big(\frac{t}{r_2^*}g(V^{\kappa}(z))\Big)dz\\
         &\ \ \ \ \ +\frac{1}{t}\int_{\R^2}\partial_{s}J_{G}\Big(\frac{t}{r_2^*}g(V^{\kappa}(z))\Big)g(V^{\kappa}(z))dz.
     \end{split}
 \end{equation*}
For brevity, let's define $L(s):=\partial_{s}J_{G}(s)s-J_{G}(s)$, and it's also a non-decreasing function. By definition \eqref{3-6} of $r_2^*$, we can abbreviate it as follows
 \begin{equation*}
     \begin{split}
     \frac{d \mathcal{W}_2(t)}{dt}&=-\frac{\kappa^2}{4\pi t}+\kappa W+\frac{r_2^*}{t^2}\int_{\R^2}L\Big(\frac{t}{r_2^*}g(V^{\kappa}(z))\Big)dz\\
     &=-\frac{\kappa^2}{4\pi t}+\kappa W+\frac{r_2^*}{t^2}\int_{\R^2}G(\partial_{s}J_{G}(\frac{t}{r_2^*}g(V^{\kappa}(z))))dz\\
     &=\frac{1}{t^2}\Big(N(t)-N(r_2^*)\Big),
 \end{split}
 \end{equation*}
where  
\begin{equation*}
    N(t)=-\frac{\kappa^2}{4\pi}t+\kappa Wt^2+r_2^*\int_{\R^2}G(\partial_{s}J_{G}(\frac{t}{r_2^*}g(V^{\kappa}(z))))dz.
\end{equation*}
We need to determine whether $N$ is monotonic at $r_2^*$, and we can compute the derivative of $N$ as follows
\begin{equation*}
    \begin{split}
        N'(t)=-\frac{\kappa^2}{4\pi}+2\kappa Wt+\int_{\R^2}g(\partial_{s}J_{G}(\frac{t}{r_2^*}g(V^{\kappa})))\partial_{ss}J_{G}(\frac{t}{r_2^*}g(V^{\kappa}))g(V^{\kappa})dz.\\
    \end{split}
\end{equation*}
Thus we have
\begin{equation}\label{C_choose}
    N'(r_2^*)=\int_{\R^2}g(V^{\kappa})\partial_{ss}J_{G}(g(V^{\kappa}))g(V^{\kappa})dz>0.
\end{equation}
According to the above formula, we can find a constant $0<c_1<1$ , such that $N'(t)>0$ on the set $[c_1r_2^*,r_2^*/c_1]$.
 Using the sign of the derivative of $\mathcal{W}_2$, we obtain $\mathcal{W}_2(r_2^{*})=\min_{t\in [c_1 r_2^*,r_2^*/c_1]}\mathcal{W}_2(t)$  and $r_2^*$ is the only minimum point in this range.

Next, we want to obtain the specific position of vortex  convergence.
\begin{lemma}\label{le3.6}
    One has
    \begin{equation*}
        dist\left(supp(\zeta^\varepsilon), b_1\right)\to 0, \ \ \text{as}\ \varepsilon \to 0,
    \end{equation*}
where $b_1=r_2^*\mathbf{e}_2$.
\end{lemma}
\begin{proof}
    Take $(0,x^\varepsilon_2)\in \text{supp}(\zeta^\varepsilon)$. Suppose (up to a subsequence) $x_2^\varepsilon\to \bar{r}\in [c_1r_2^*,r_2^*/c_1]$, as $\varepsilon \to 0$. We now show that $\bar{r}=r_2^{*}$. Set
    $\tilde{\zeta}^\varepsilon=\zeta^\varepsilon(\cdot-r_2^{*}\mathbf{e}_2+\bar{r}\mathbf{e}_2)\in \mathcal{A}_{\varepsilon, \Lambda_0}$. Then
    \begin{equation*}
        \int_D \int_D{\tilde{\zeta}^\varepsilon(x)\log\frac{1}{|x-x'|}\tilde{\zeta}^\varepsilon}(x')dxdx'= \int_D \int_D{{\zeta}^\varepsilon(x)\log\frac{1}{|x-x'|}{\zeta}^\varepsilon}(x')dxdx'.
    \end{equation*}
    By Corollary \ref{lem3.5} , we have
     \begin{equation*}
        \begin{split}          
                   \frac{1}{\varepsilon^2}\int_{D}J(x_2,\varepsilon^2 \zeta^{\varepsilon} )dx&=\int_{\Pi^{\varepsilon}}J(x_2^{\varepsilon}+\varepsilon y_2,\varepsilon^2 \zeta^{\varepsilon}(x^\varepsilon+\varepsilon y)) dy\\
            &=\int_{\Pi^{\varepsilon}}\sup_{t}(\varepsilon^2 \zeta^{\varepsilon}(x^\varepsilon+\varepsilon y)t-\Big[\int_{0}^{t}\delta_{\varepsilon}f(s)+(x_{2}^{\varepsilon}+\varepsilon y_2)g(s)ds\Big])dy\\
            &=\int_{\R^2}\sup_{t}(\bar{r}g(V
            ^{\kappa}(\sqrt{\bar{r}}y))t-\bar{r}G(t))dy+o(1)\\
            &=\int_{\R^2}J_{G}(g(V^{\kappa}(z)))dz+o(1).
        \end{split}
    \end{equation*}
    Similarly, we can obtain
    \begin{equation*}
        \frac{1}{\varepsilon^2}\int_{D}J(x_2,\varepsilon^2 \tilde{\zeta}^{\varepsilon}(x) )dx=\frac{r_2^*}{\bar{r}}\int_{\R^2}J_{G}\Big(\frac{\bar{r}}{r_2^*}g(V^{\kappa}(z))\Big)dz+o(1).
    \end{equation*}
    Since $\mathcal{E}_\varepsilon(\tilde{\zeta}^\varepsilon)\le \mathcal{E}_\varepsilon(\zeta^\varepsilon)$, it follows that
    \begin{equation*}
        \begin{split}
            &\frac{1}{4\pi} \int_D \int_D{{\zeta}^\varepsilon(x)\log\frac{1}{|\bar{x}-x'|}{\zeta}^\varepsilon}(x')dxdx'+W\int_{D}x_2{\zeta}^\varepsilon dx +\frac{1}{\varepsilon^2}\int_{D}J(x_2,\varepsilon^2 \zeta^{\varepsilon})dx  \\
            &\ \ \ \  \le \frac{1}{4\pi} \int_D \int_D{\tilde{\zeta}^\varepsilon(x)\log\frac{1}{|\bar{x}-x'|}\tilde{\zeta}^\varepsilon}(x')dxdx'+W\int_{D}x_2\tilde{\zeta}^\varepsilon dx+\frac{1}{\varepsilon^2}\int_{D}J(x_2,\varepsilon^2 \tilde{\zeta^{\varepsilon}})dx.
        \end{split}
    \end{equation*}
    Letting $\varepsilon \to 0$, we get
    \begin{equation*}
        \mathcal{W}_2(\bar{r})\le \mathcal{W}_2(r_2^{*}).
    \end{equation*}
Because $\bar{r}\in[c_1r_2^*,r_2^*/c_1]$, and $\mathcal{W}_2(r_2^{*})=\min_{t\in[c_1r_2^*,r_2^*/c_1] }\mathcal{W}_2(t)$, which implies
 $\bar{r}=r_2^{*}$ and the proof is thus complete.
\end{proof}

\begin{proof}[Proof of Theorem \ref{thm3}]
   Combining Lemma \ref{le2.1} to Lemma \ref{le2.8} in section \ref{sect2} and Lemma \ref{le3.6} in section \ref{sect3}, we can obtain Theorem \ref{thm3} by taking $\lambda=\frac{1}{\varepsilon^2}$. 
\end{proof}

{\bf Acknowledgments.}
{ This work was supported by NNSF of China Grant 11831009.
}

\end{document}